\documentclass{arxivarticle}

\usepackage[square, comma, numbers, sort&compress]{natbib}
\usepackage{mathHeader}
\figpath{figures/}

\title{On Analytical and Topological Properties of Separatrices in 1-D Holomorphic Dynamical Systems and Complex-Time Newton Flows}
\author{
	Marcus~Heitel\\ 
	Institute for Numerical Mathematics\\
	Ulm University\\
	Ulm, Germany
	\texorpdfstring{\And}{and}
	Dirk~Lebiedz \\
	Institute for Numerical Mathematics\\
	Ulm University\\
	Ulm, Germany
}

\begin{document}
\maketitle

\begin{abstract}
    Separatrices divide the phase space of some ho\-lo\-morphic dynamical systems into separate basins of attraction or  \enquote{stability regions} for distinct fixed points. \enquote{Bundling} (high density) and mutual \enquote{repulsion} of trajectories are often observed at separatrices in phase portraits, but their global mathematical characterisation is a difficult problem. 
	For 1-D complex polynomial dynamical systems we prove the existence of a separatrix for each critical point at infinity via transformation to the Poincar{\'e} sphere. We show that introduction of complex time allows a significantly extended view with the study of corresponding Riemann surface solutions, their topology, geometry and their bifurcations/ramifications related to separatrices. We build a bridge to the Riemann $\xi$-function and present a polynomial approximation of its Newton flow solution manifold with precision depending on the polynomial degree. 
\end{abstract}






\section{Introduction}
In this paper, we analyse properties of special phase space structures of the one{\hyp}dimensional complex autonomous dynamical system
\begin{equation}\label{eq:ode}
	z' = \dd[z]{t} = f(z), \quad z(0)=z_0 \in \C
\end{equation}
where $f:\C \ra \C$ is a holomorphic function. Many studies have been performed for such systems, e.g.\ see \cite{Broughan2003, Broughan2003a} and for corresponding Newton flows, e.g.\ see \cite{Neuberger2014}
\begin{equation}\label{eq:newtonflow}
	z' = -\frac{f(z)}{f'(z)}, \quad z(0)=z_0 \in \C,
\end{equation}
also and in particular in a context where $f(z)$ is number{\hyp}theoretically significant like the Riemann $\zeta(z)${\hyp} or $\xi(z)${\hyp}functions.
As a fundamental result it has been shown that holomorphic flows of type (\ref{eq:ode}) do not allow limit cycles and sectors of the flow around equilibria are confined by separatrices (see \cite{Broughan2003}).
The recent work of Schleich et al.~\cite{Schleich2018} suggests that the behaviour of $\xi$ Newton flow separatrices at infinity may contain information about the exact location of the equilibria of such flows, e.g.\ the non{\hyp}trivial zeroes of $\xi(z)$. 

The main focus of our article is to investigate holomorphic flows at infinity and identify analytical and topological properties of separatrices dividing two sectors around centres.
Mapping the complex phase space to the Poincar{\'e} sphere provides insight into the dynamics of the flow at infinity \cite{Perko2000}. Other transformations can reveal characteristics at infinity as well (see \cref{sec:infinity}). The global phase space behaviour of polynomial vector fields $f$ is well studied. We prove that critical points at infinity must have at least one incoming or outgoing trajectory which is a separatrix. Thus, separatrices are closely related to heteroclinic or homoclinic orbits connecting to critical points at infinity. This refers to the work of Davis and Skodje~\cite{Davis1999} and Al{\hyp}Khateeb et al.~\cite{Al-Khateeb2009} for dissipative real functions $f:\R^n \ra \R^n$ and their slow invariant manifolds. 

Broughan systematically analyses holomorphic flows in a couple of articles \cite{Broughan2003, Broughan2003a, Broughan2004, Broughan2005}. He refers to the set composed of a finite equilibrium and its characteristic orbits as the \emph{neighbourhood} of the corresponding equilibrium. Examples are a centre and all of its encircling periodic orbits or a stable node and its basin of attraction. The boundary of the neighbourhood of centres is the union of separatrices (=boundary components).
While Broughan considers separatrices primarily from an analytical point of view, we consider them as global objects on the manifold of solutions and deal with them topologically as well. In particular, the winding number of periodic orbits around their corresponding centres changes when crossing a separatrix. 

\Cref{sec:separatrices} gives a brief introduction into the topics of holomorphic flows and separatrices. A suitable definition of the separatrix is given as well. 
In \cref{sec:infinity}, analytical basics for the computation of critical points at infinity are introduced. We give a brief introduction to the Poincar{\'e} sphere and other transformations that make the analysis at infinity possible. It is shown that for certain polynomial dynamical systems, (limit) cycles cannot exist at infinity (and finite limit cycles cannot exist as well, cf. \cite{Broughan2003}). Moreover, equilibria at infinity and separatrices are directly correlated. 
\Cref{sec:topology} deals with the topology of centres and their periodic orbits. Since the right{\hyp}hand side of the ODE \eqref{eq:ode} is smooth, periodic orbits cannot intersect. Their winding number is $\pm 1$ depending on which side of the separatrix the orbits are located. 

\section{Separatrices}\label{sec:separatrices}
When considering holomorphic flows, e.g.\ $z' = \cosh\big(z-\frac{1}{2}\big)$ (see \Cref{fig:cosh}), it becomes evident that there are special trajectories which divide the whole bunch of trajectories into different regions of the same qualitative stability behaviour. Tra\-jectories in the neighbourhood bundle at these trajectories making the phase portrait look similar to slow invariant manifolds (see e.g.\ \cite{Heiter2018,Lebiedz2016,Lebiedz2011}).

\begin{figure}[htp]
	\begin{center}
		\includegraphics[width=0.7\linewidth]{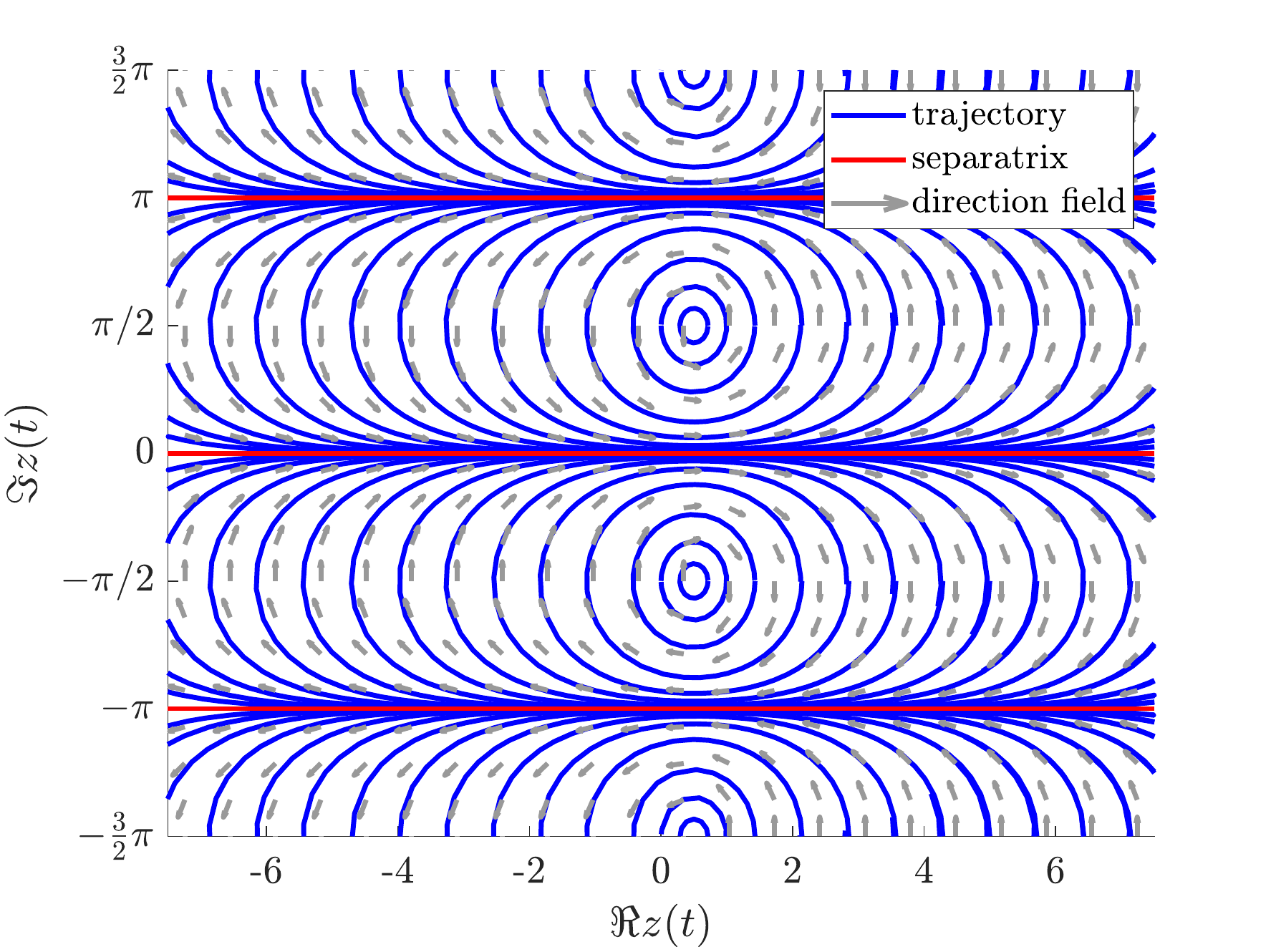}
		\caption{Holomorphic flow of $z'= \cosh(z-1/2)$. \enquote{Ordinary} trajectories in blue, separatrix in red, direction field (= normalized vector field) in grey.}\label{fig:cosh}
	\end{center}
\end{figure}

In \Cref{fig:cosh}, these trajectories are the lines $\Im z = k\pi, k \in \Z$ and the corresponding regions are the sets $S_k := \left\{z  \in \C \, : \, \Im z \in \big(k\pi, (k+1)\pi\big) \right\}$ for $k\in \Z$. Within such a region $S_k$ all trajectories are periodic orbits around the equilibrium $z_k^* := \frac{1}{2} + \left(k + \frac{1}{2}\right)\pi$. We call those separating lines \emph{separatrices}.
However, an appropriate and precise mathematical definition of separatrices is a controversial topic (cf. \cite{Broughan2003a}, Section 2 for a brief overview of definitions). In the context of holomorphic flows, the following Definition~\ref{def:sep} is reasonable and well suited as we will see later. 
\begin{definition}[Separatrix \cite{Broughan2003a}]\label{def:sep}
	A trajectory $\gamma$ is a \emph{positive (negative) separatrix} for the holomorphic flow \eqref{eq:ode} if there is a point $z \in \gamma$ such that the maximum interval of existence of the path starting at $z$ and following the holomorphic flow \eqref{eq:ode} in positive (negative) time is finite. 
\end{definition}
Although, this definition is not based on the separation of stability regions in the phase space, Broughan~\cite{Broughan2003, Broughan2003a} shows that the boundary of these stability regions consists of trajectories with finite maximum interval of existence, and thus separatrices.  

Since separatrices have finite maximum interval of existence, the vector field $f$ cannot be bounded on separatrices. Otherwise, one could enlarge the interval of existence. Therefore, a positive separatrix for entire functions $f$ must reach out to infinity in phase space. This motivates the investigation of the qualitative behaviour of holomorphic flows at infinity.

\section{Critical Points at Infinity}\label{sec:infinity}
In this section, techniques for investigation of the behaviour of dynamical systems at infinity are presented. We follow the introductions of Perko~\cite{Perko2000}, Meiss~\cite{Meiss2007}, Roeder~\cite{Roeder2003}, Gingold and Solomon~\cite{Gingold2013} into this topic. We use the identification $\C \simeq \R^2$ and consider two{\hyp}dimensional real dynamical systems. 

\paragraph{Polar Coordinates}
There are several possibilities to compactify and analyse complex{\hyp}valued dynamical systems. One obvious transformation uses polar coordinates $(r,\theta)$ and regards the limit $r\ra \infty$ (cf. \cite{Roeder2003}). 
Consider a polynomial, two{\hyp}dimensional real dynamical system of the form
\begin{equation}
\begin{array}{lll}\label{eq:2DPolySystem}
	x' &= \dd[x]{t} &=P(x,y)= P_0(x,y) +P_1(x,y) + \cdots + P_d(x,y)\\
	y'&=  \dd[y]{t} &=Q(x,y) = Q_0(x,y) +Q_1(x,y) + \cdots + Q_d(x,y),
\end{array}
\end{equation}
where $P$ and $Q$ are polynomials in $x$ and $y$ with $P_i(x,y)$ resp. $Q_i(x,y)$ being the homogeneous polynomial of degree $i$ in $P$ resp.\ $Q$. Using polar coordinates, the following system is obtained:
\begin{equation}
	\begin{aligned}\label{eq:2DPolarCoord}
		r' & = \cos \theta \, P(r\cos \theta, r\sin \theta) + \sin \theta \, Q(r\cos \theta, r\sin \theta) \\
		& = \sum\limits_{k=0}^{d} r^k \left(\cos \theta P_k(\cos \theta, \sin \theta) + \sin \theta Q_k(\cos \theta, \sin \theta) \right) \\
		\theta' &= \frac{1}{r}\left(-\sin \theta  P(r\cos \theta, r\sin \theta) + \cos \theta Q(r\cos \theta, r\sin \theta)\right)\\
		& = \frac{1}{r}\sum\limits_{k=0}^{d} r^k \left(-\sin \theta P_k(\cos \theta, \sin \theta) + \cos \theta Q_k(\cos \theta,\sin \theta) \right)
	\end{aligned}
\end{equation}
For taking the limit $r \ra \infty$, $r$ is replaced by $s = \frac{1}{r}$ and then the limit $s \ra 0+$ is considered
\begin{subequations}\label{eq:2DInvPolarCoord}
	\begin{align}
	s' &  = -s^2\sum\limits_{k=0}^{d} s^{-k} \left(\cos \theta P_k(\cos \theta, \sin \theta) + \sin \theta Q_k(\cos \theta, \sin \theta) \right) \label{eq:ds}\\
	\theta' & = s\sum\limits_{k=0}^{d} s^{-k} \left(-\sin \theta P_k(\cos \theta, \sin \theta) + \cos \theta Q_k(\cos \theta,\sin \theta) \right) \label{eq:dtheta}. 
	\end{align}
\end{subequations}
We denote the degrees of the right{\hyp}hand side of \eqref{eq:ds} in $s$ by $2-I$ and of \eqref{eq:dtheta} by $1-J$. As we will see, behaviour at infinity depends on the number $\hat{k}:=I-J$.

\paragraph{Poincar\'{e} Sphere}
In the transformation \eqref{eq:2DInvPolarCoord}, infinity corresponds to $s=0$. Handling of a fixed point at infinity might be complicated since infinity is just a single point here. Thus, the study of separatrices is hard if they accumulate at this single point of infinity.
An alternative is the projection to the Poincar\'{e} sphere which distinguishes the behaviour at infinity. 
To project a point $(x^*,y^*)$ of the dynamical system \eqref{eq:2DPolySystem} onto a point $(X^*,Y^*,Z^*)$ on the northern hemisphere
\begin{equation*}
	S^{2}_{+} = \left\{(X,Y,Z) : X^2+Y^2+Z^2=1, Z\ge 0\right\},
\end{equation*}
the $(x,y)$ plane is pinned in 3D coordinates to the plane $(x,y,1)$. The corresponding point $(X^*,Y^*,Z^*)$ is then obtained by calculating the intersection point of the straight line through the origin and $(x^*,y^*,1)$, and the surface of the northern hemisphere as shown in \cref{fig:PoincareSphere}.

\begin{figure}[ht]
	\centering
\begingroup%
  \makeatletter%
  \providecommand\color[2][]{%
    \errmessage{(Inkscape) Color is used for the text in Inkscape, but the package 'color.sty' is not loaded}%
    \renewcommand\color[2][]{}%
  }%
  \providecommand\transparent[1]{%
    \errmessage{(Inkscape) Transparency is used (non-zero) for the text in Inkscape, but the package 'transparent.sty' is not loaded}%
    \renewcommand\transparent[1]{}%
  }%
  \providecommand\rotatebox[2]{#2}%
  \newcommand*\fsize{\dimexpr\f@size pt\relax}%
  \newcommand*\lineheight[1]{\fontsize{\fsize}{#1\fsize}\selectfont}%
  \ifx\svgwidth\undefined%
    \setlength{\unitlength}{261.25442733bp}%
    \ifx\svgscale\undefined%
      \relax%
    \else%
      \setlength{\unitlength}{\unitlength * \real{\svgscale}}%
    \fi%
  \else%
    \setlength{\unitlength}{\svgwidth}%
  \fi%
  \global\let\svgwidth\undefined%
  \global\let\svgscale\undefined%
  \makeatother%
  \begin{picture}(1,0.98978361)%
    \lineheight{1}%
    \setlength\tabcolsep{0pt}%
    \put(0,0){\includegraphics[width=\unitlength,page=1]{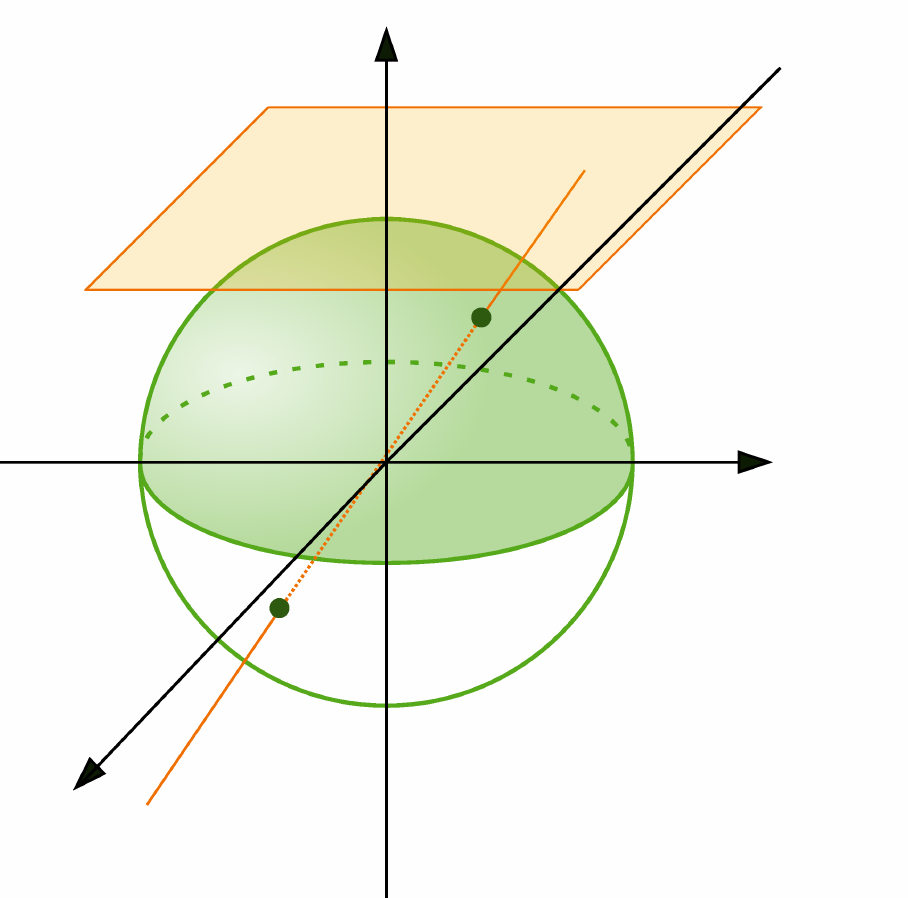}}%
    \put(0.03735657,0.09338544){\color[rgb]{0,0,0}\makebox(0,0)[lt]{\lineheight{1.25}\smash{\begin{tabular}[t]{l}$X$\end{tabular}}}}%
    \put(0.85926585,0.46932151){\color[rgb]{0,0,0}\makebox(0,0)[lt]{\lineheight{1.25}\smash{\begin{tabular}[t]{l}$Y$\end{tabular}}}}%
    \put(0.40992321,0.96651845){\color[rgb]{0,0,0}\makebox(0,0)[lt]{\lineheight{1.25}\smash{\begin{tabular}[t]{l}$Z$\end{tabular}}}}%
    \put(0,0){\includegraphics[width=\unitlength,page=2]{PoincareSphere.pdf}}%
    \put(0.68061247,0.65529413){\color[rgb]{0,0,0}\makebox(0,0)[lt]{\lineheight{1.25}\smash{\begin{tabular}[t]{l}$W=(X^*,Y^*,Z^*)$\end{tabular}}}}%
    \put(0,0){\includegraphics[width=\unitlength,page=3]{PoincareSphere.pdf}}%
    \put(0.5507621,0.83325791){\color[rgb]{0.54901961,0.54901961,0.54901961}\makebox(0,0)[lt]{\lineheight{1.25}\smash{\begin{tabular}[t]{l}$w=(x^*,y^*)$\end{tabular}}}}%
    \put(0,0){\includegraphics[width=\unitlength,page=4]{PoincareSphere.pdf}}%
  \end{picture}%
\endgroup%

	\caption{Construction of Poincar\'{e} Sphere}\label{fig:PoincareSphere}
\end{figure}
In short mathematical notation, the transformation can be written as the following two relations:
\begin{align*}
	(X,Y,Z) &= \left(\frac{x}{\sqrt{1+x^2+y^2}},\frac{y}{\sqrt{1+x^2+y^2}},\frac{1}{\sqrt{1+x^2+y^2}} \right)\\
	(x,y) &= \left(\frac{X}{Z},\frac{Y}{Z}\right).
\end{align*}

Simple calculations show that the dynamics on the Poincar\'{e} sphere are given by
\begin{equation*}
	\begin{aligned}
	 	X' &= Z\left((1-X^2)P\big(\frac{X}{Z},\frac{Y}{Z}\big)-XYQ\big(\frac{X}{Z},\frac{Y}{Z}\big)\right)\\
	 	Y' &= Z\left((1-Y^2)Q\big(\frac{X}{Z},\frac{Y}{Z}\big)-XYP\big(\frac{X}{Z},\frac{Y}{Z}\big)\right)\\
	 	Z' &= -Z^2\left(XP\big(\frac{X}{Z},\frac{Y}{Z}\big)+YQ\big(\frac{X}{Z},\frac{Y}{Z}\big)\right)
	\end{aligned}
\end{equation*}
In the following, the degree of the polynomials $P$ and $Q$ is assumed to be a natural number $d \in N$. Since we are interested in the case $Z=0$ which corresponds to infinity and $P\big(\frac{X}{Z},\frac{Y}{Z}\big) = \mathcal{O}(Z^{-d})$, negative exponents are prevented by transforming time $t$ into $\tau$, where $\frac{\dif t}{\dif \tau}=Z^{d-1}$. Additionally, evaluating the obtained differential equations for $\tau$ at $Z=0$ yields the dynamics on the equator
\begin{equation}\label{eq:dynamicsEquator}
	\begin{aligned}
		\dot{X} = \frac{\dif X}{\dif \tau} =& -Y\left(XQ_d(X,Y)-YP_d(X,Y)\right) \, \, \, \vspace*{1mm} \\
		\dot{Y} = \frac{\dif Y}{\dif \tau} =& X\left(XQ_d(X,Y)-YP_d(X,Y)\right),
	\end{aligned}
\end{equation}
where  $P_d(X,Y)$ and $Q_d(X,Y)$ are the terms with exactly degree $d$ in $P$ and $Q$. Thus, fixed points at infinity must satisfy the following equation
\begin{equation*}
	XQ_d(X,Y)-YP_d(X,Y) = 0.
\end{equation*}

Since this equation is homogeneous of degree $d+1$, an equilibrium $(X,Y,0)$ induces the antipodal point $(-X,-Y,0)$ to be a fixed point as well. The antipodal point is the same type of fixed point. However, the direction of the flow reverses if the natural number $d$ is even.

To characterise the type of an equilibrium $(X,Y,0)$ at infinity, it can be useful to transform the system in the neighbourhood of the equilibrium to the tangent plane at this fixed point.
For instance, if $X\neq 0$, the mapping
\[ \beta = \frac{Y}{X}, \gamma = \frac{Z}{X}\]
projects $(X,Y,Z)$ onto the $(\beta,\gamma)${\hyp}plane and corresponding dynamics with respect to a rescaled time $\dif s = X^{d-1} \dif \tau$
\begin{equation}\label{eq:dynamicsBetaGamma}
	\begin{array}{lll}
		\dot{\beta} &= \frac{\dif \beta}{\dif s} =&\gamma^d Q\left(\frac{1}{\gamma},\frac{\beta}{\gamma}\right)-\beta \gamma^d P\left(\frac{1}{\gamma},\frac{\beta}{\gamma}\right) \, \, \,  \\
		\dot{\gamma} &= \frac{\dif \gamma}{\dif s} =&  -\gamma^{d+1} P\left(\frac{1}{\gamma},\frac{\beta}{\gamma}\right).
	\end{array}
\end{equation}

A fixed point $(\beta_0, 0)$ in this new system \eqref{eq:dynamicsBetaGamma} corresponds to a fixed point $\pm \left(\frac{1}{\sqrt{1+\beta_0^2}},\frac{\beta_0}{\sqrt{1+\beta_0^2}},0\right)$ on the Poincar{\'e} sphere.

\paragraph{Compactification}
Finally, a third possibility for the compactification might be the following approach (cf. \cite{Gingold2013}), where a point $\bar{x} \in \R^2$ is mapped to the point 
\begin{equation*}
	u = T(\bar{x}) := \frac{2\bar{x}}{1+\sqrt{1+4\mynorm{\bar{x}}}}
\end{equation*}
on the unit ball $U:=U_1(0) = \left\{u \in \R^2: \mynorm{u} <1  \right\}$. Its inverse mapping is given by 
\begin{equation*}
	\bar{x} = \frac{u}{1-\mynorm{u}^2}.
\end{equation*}
This mapping can be continued to the boundary $\partial U$, which is mapped to a set of points $p_\infty$ at infinity that differ in their \enquote{direction} $p$ at infinity. Formally, it is the set 
\begin{equation*}
	\mathrm{ID} := \left\{p_\infty \, : \, \mynorm{p} = 1 \right\},
\end{equation*}
where the following Definition~\ref{def:InfConvergence} clarifies what is meant by $p_\infty$ and its corresponding direction $p$ to infinity.
\begin{definition}[cf. \cite{Gingold2013}]\label{def:InfConvergence}
	A continuous function $y:(t_\mathrm{min},t_\mathrm{max}) \ra \R^n$ \emph{diverges in direction $p \in \partial U$ to infinity or diverges to $p_\infty$}, in notation $\lim\limits_{t \ra t_\mathrm{max}-} y(t) = p_\infty$, iff 
	\begin{align*}
		\lim\limits_{t\ra t_\mathrm{max}-} \mynorm{y(t)} &= \infty \quad \mathrm{and}\\
		\lim\limits_{t \ra t_\mathrm{max}-} \frac{y(t)}{\mynorm{y(t)}} &= p.
	\end{align*}
\end{definition}
Thus $T$ maps $\bar{\R}^2 := \R^2 \cup \mathrm{ID}$ bijective onto $\bar{U} = U \cup \partial U$.
\begin{remark}
	$\bar{\R}^2$ can be equipped with the chordal metric in order to become a complete metric space. See \cite{Gingold2013}.
\end{remark}

Compactification of the dynamical system \eqref{eq:2DPolySystem} with mapping $T$ yields
\begin{equation*}
    u' = \mymatrix{u_1'\\u_2'}= \frac{(1+\mynorm{u}^2)I-2uu\tp}{(1+\mynorm{u}^2){(1-\mynorm{u}^2)}^{-1}}\mymatrix{P\left(\frac{u_1}{1-\mynorm{u}^2},\frac{u_2}{1-\mynorm{u}^2}\right) \\Q\left(\frac{u_1}{1-\mynorm{u}^2},\frac{u_2}{1-\mynorm{u}^2}\right)}
\end{equation*}
Again transforming the time $t$ to $\tau$ with $\dif t = (1+\mynorm{u}^2){(1-\mynorm{u}^2)}^{d-1} \dif \tau$ results in the dynamical system
\begin{equation}\label{eq:2DCompact}
	\qquad	\dot{u} = \frac{\dif u}{\dif \tau} = {(1-\mynorm{u}^2)}^{d}\left[(1+\mynorm{u}^2)\mymatrix{P(\cdot,\cdot) \\Q(\cdot,\cdot)} - 2\left(u\tp \mymatrix{P(\cdot,\cdot) \\Q(\cdot,\cdot)}\right) \cdot u\right],
\end{equation}
where $P$ and $Q$ are evaluated at $\left(\frac{u_1}{1-\mynorm{u}^2},\frac{u_2}{1-\mynorm{u}^2}\right)$.
Finally, for $u \in \partial U$, we obtain
\begin{equation*}
	\begin{aligned}\label{eq:2DCompactInf}
	\dot{u_1} 	&= 2 P_d(u_1,u_2) - 2\left( u_1 P_d(u_1,u_2) + u_2Q_d(u_1,u_2) \right) u_1\\
	& = 2(1-u_1^2)P_d(u_1,u_2) - 2u_2u_2 Q_d(u_1,u_2)\\
	\dot{u_2} 	&= 2 Q_d(u_1,u_2) - 2\left( u_1 P_d(u_1,u_2) + u_2Q_d(u_1,u_2) \right) u_2 \\
	& =  -2u_1u_2 P_d(u_1,u_2) + 2(1-u_2^2)Q_d(u_1,u_2),
	\end{aligned}
\end{equation*}
which is twice the right{\hyp}hand side of \eqref{eq:dynamicsEquator}. So the compactification via mapping $T$ is very similar to the transformations to the Poincar{\'e} sphere and subsequently projection to the plane containing the equator of the sphere.

\subsection{Definition of Critical Points at Infinity}
Critical points at infinity may be introduced with respect to the Poincar\'e sphere as
\begin{definition}[Critical Point at Infinity]
	A \emph{critical point at infinity} for the system \eqref{eq:2DPolySystem} is a point $p \in \partial U$  such that
	\begin{equation*}
		p_1 Q_d(p_1,p_2) - p_2 P_d(p_1,p_2) = 0.
	\end{equation*}
\end{definition}

\begin{proposition}\label{prop:InfPointCrit}
	Let $p$ be in $\partial U$. Then the following statements are equivalent
	\be[(a)]
		\item\label{prop:Text}$p$ is a critical point at infinity for the system \eqref{eq:2DPolySystem},
		\item\label{prop:Def}$p_1 Q_d(p_1,p_2) - p_2 P_d(p_1,p_2) = 0$,
		\item\label{prop:LinDep}There is a real number $\alpha \in \R$ such that $\mymatrix{P_d(p_1,p_2)\\ Q_d(p_1,p_2)} = \alpha \mymatrix{ p_1\\ p_2}$,
		\item\label{prop:Compact}$\mymatrix{P_d(p_1,p_2)\\ Q_d(p_1,p_2)} =  \mymatrix{p_1\\ p_2}\tp 	\mymatrix{P_d(p_1,p_2)\\ Q_d(p_1,p_2)} \mymatrix{p_1\\ p_2}$ --- this statement is used for definition of critical points at infinity in \cite{Gingold2013}.
	\ee
\end{proposition}
\begin{proof}
	\ref{prop:Text} and \ref{prop:Def} are equivalent by definition. \ref{prop:Def} is equivalent to \[\det\mymatrix[rr]{p_1 & P_d(p_1,p_2)\\ p_2 & Q_d(p_1,p_2)} = 0.\]
	This holds exactly when \ref{prop:LinDep} holds. \ref{prop:Compact} is an equation of the form $y=p\tp yp = p p\tp y$, which is an eigenvalue problem of the matrix $pp\tp$ of rank 1 with eigenvalue $1$. Since $p = p p\tp p$, the vector $p$ lies in the image of the matrix  $p p\tp$. Now it follows
	\begin{equation*}
		\left\{y \in \R^2 : y = p\tp y p \right\} = \left\{\alpha p : \alpha \in \R \right\}.
	\end{equation*}
	Therefore, \ref{prop:LinDep} and \ref{prop:Compact} are equivalent.
\end{proof}

\subsection{Heteroclinic Orbits as Separatrix Generators}
For vector fields $f$ arising in the context of two{\hyp}time{\hyp}scale systems, so called \emph{slow invariant manifolds (SIMs)} are special trajectories sharing a phase portrait similar to separatrices. In regions around those SIMs, other trajectories bundle onto these SIMs on the fast time scale. Once trajectories are close to them, they stay close to them. 

Al{\hyp}Khateeb et al.~\cite{Al-Khateeb2009} conjecture and make plausible that the one{\hyp}dimensional SIMs are heteroclinic orbits that start in critical points (finite or infinite) of saddle type with exactly one unstable direction (based on \cite{Davis1999}). Therefore, it seems natural to investigate whether separatrices are heteroclinic, homoclinic orbits or none of them. 

In the following we restrict ourselves to complex polynomial differential systems.
\begin{lemma}\label{lem:complexPoly}
	Let $f(z) = \sum\limits_{k=0}^{d} \alpha_k z^k$ be a complex polynomial of degree $d\ge 1$, i.e., $\alpha_d \neq 0$ and $\alpha_k = a_k + \rmi b_k \in \C, \alpha_k, \beta_k \in \R$ for all $k=0,1,\ldots, d$. Then, the differential equation $z' = f(z)$ can be identified in $\R^2$ by the differential system
	\begin{equation}\label{eq:2DholPolySystem}
		\begin{aligned}
			x' &= P(x,y) = \sum\limits_{k=0}^{d} P_k(x,y) := \sum\limits_{k=0}^{d} a_k \xi_k(x,y) - b_k\eta_k(x,y)\\
			y' &= Q(x,y) = \sum\limits_{k=0}^{d} Q_k(x,y) := \sum\limits_{k=0}^{d} a_k \eta_k(x,y) + b_k\xi_k(x,y),
		\end{aligned}
	\end{equation}
	where 
	\begin{equation*}
		\xi_k(x,y):= \sum_{\ell=0}^{\lfloor k/2\rfloor}{\binom{k}{2\ell}}(-1)^\ell x^{k-2\ell}y^{2\ell}, \quad \eta_k(x,y):= \sum_{\ell=0}^{\lfloor (k-1)/2\rfloor}{\binom{k}{2\ell +1}}{(-1)}^\ell x^{k-2\ell-1}y^{2\ell+1}.
	\end{equation*}
	Furthermore, if the equator of the Poincar{\'e} sphere consists of infinitely many critical points, then $d$ is an odd number and $\alpha_d = a_d $ is a real number.
\end{lemma}
\begin{proof}
	Let $f(z) = \sum\limits_{k=0}^{d} \alpha_k z^k$ be a complex polynomial of degree $d\ge1$. Because of 
	\[ {(x+iy)}^k = \xi_k(x,y) + \rmi \, \eta_k(x,y), \]
	the transformation \eqref{eq:2DholPolySystem} is clear.  
	We consider \eqref{eq:2DPolarCoord} and determine $I\in \left\{1,2,\ldots,d\right\}$ and $J\in \left\{-1,0,\ldots, d-1\right\}$ in this special case. First, notice that $\xi_k(1,0)=1$ and $\eta_k(1,0) = 0$. If $k$ is an even number, then $\xi_k(0,1)=(-1)^{\frac{k}{2}}$ and $\eta_k(0,1)=0$. If $k$ is an odd number, then $\xi_k(0,1)=0$ and $\eta_k(0,1)=(-1)^{\frac{k-1}{2}}$. 
	Inserting $\theta = 0$ in \eqref{eq:2DPolarCoord}, yields 
	\begin{eqnarray*}
		r'=	&  \sum_{k=0}^{d} r^k P_k(1,0) = \sum_{k=0}^{d} r^k a_k \\
		\theta'= & \frac{1}{r}\sum_{k=0}^{r} r^k  Q_k(1,0) = \frac{1}{r}\sum_{k=0}^{r} r^k  b_k.
	\end{eqnarray*}
	If $a_d\neq 0$ then $I=d$. If $b_d \neq 0$, $J=d-1$.
	
	Inserting $\theta = \pi/2$ in \eqref{eq:2DPolarCoord}, yields 
	\begin{eqnarray*}
		r'=	& \sum_{k=0}^{d} r^k Q_k(0,1) = \sum_{k=0}^{d} r^k (a_k\eta_k(0,1)+b_k\xi_k(0,1))\\
		\theta'= & -\frac{1}{r}\sum_{k=0}^{d} r^k  P_k(0,1) = -\frac{1}{r}\sum_{k=0}^{d}  r^k  (a_k\xi_k(0,1)-b_k\eta_k(0,1))
	\end{eqnarray*}
	If $d$ is even, we obtain 
	\begin{eqnarray*}
		r'=	& r^d b_d{(-1)}^{\frac{d}{2}}  + \sum_{k=0}^{d-1} r^k (a_k\eta_k(0,1)+b_k\xi_k(0,1))\\
		\theta'= &  -\frac{1}{r}\left(r^d a_d{(-1)}^{\frac{d}{2}} + \sum_{k=0}^{d-1}  r^k  (a_k\xi_k(0,1)-b_k\eta_k(0,1)) \right)
	\end{eqnarray*}
		
	\Cref{tab:k=I-J} concludes which numerical values $\hat{k}=I-J$ may assign.
	\begin{table}[htp]
		\caption{values for $\hat{k}=I-J$ with respect to $d$ and $\alpha_d$}
		\label{tab:k=I-J}
		\begin{center}
			\def\arraystretch{1.5}
			\begin{tabular}{l|lll}
				& $a_d=0$ & $b_d=0$  & $a_d \cdot b_d \neq 0$ \\ \hline
				$d$ even & $\hat{k}=d-(d-1)=1$ & $\hat{k}=d-(d-1)=1$ & $\hat{k}=d-(d-1)=1$\\
				$d$ odd & $\hat{k}=I-(d-1)\le 1$ & $\hat{k}=d-J\ge1$ & $\hat{k}=d-(d-1)=1$
			\end{tabular}
		\end{center}
	\end{table}
	This shows, that $\hat{k}\ge 2$ only if $d$ is odd and $b_d = 0$. The statement now follows applying Theorem 2 in \cite{Roeder2003}.
\end{proof}

\begin{theorem}[cf. Theorem 18 of \cite{Gingold2013}]\label{thm:ExistSep}
	Let $p \in \partial U$ be a critical point at infinity of system \eqref{eq:2DholPolySystem} with polynomial degree $d\ge2$. According to Proposition~\ref{prop:InfPointCrit} there is an $\alpha \in \R$ such that $\mymatrix{P_d(p_1,p_2)\\ Q_d(p_1,p_2)} = \alpha \mymatrix{p_1\\ p_2}$. All critical points at infinity are generic, i.e., $\alpha \neq 0$. If $\alpha > 0 \, (\alpha < 0)$, there is a separatrix that ends (starts) at the critical point $p$ at infinity.
	
\end{theorem}
\begin{proof}
	Let $p$ be a critical point at infinity with corresponding \(\alpha \in \R\). At first, we show that $\alpha \neq 0$. Assume that $\alpha = 0$. Thus it holds 
	\begin{align*}
		0 &= P_d(p_1,p_2) = a_d \xi_d(p_1,p_2) - b_d\eta_d(p_1,p_2)\\
		0 &= Q_d(p_1,p_2) = a_d \eta_d(p_1,p_2) + b_d\xi_d(p_1,p_2)
	\end{align*}
	which is equivalent to \[\mymatrix{0\\0} = \mymatrix[rr]{a_d & -b_d\\b_d & a_d} \mymatrix{\xi_d(p_1,p_2)\\ \eta_d(p_1,p_2)}.\]
	The matrix in this equation is regular. Otherwise $\alpha_d=a_d+\rmi b_d$ would be zero and using Lemma~\ref{lem:complexPoly}  yields $(p_1+\rmi p_2)^d = \xi_d(p_1,p_2) + \rmi \eta_d(p_1,p_2)= 0$. This contradicts $p\in \partial U$. The assumption $\alpha = 0$ must have been wrong.
		
	Because of Proposition~\ref{prop:InfPointCrit}, we can apply Theorem 18 resp. Theorem 14 of \cite{Gingold2013}. There it is shown that the vector field is incomplete. More precisely, the authors show that \eqref{eq:2DCompact} is equivalent to 
	\begin{equation*}
		\frac{\dif \,(u-p)}{\dif \tau} = A(u-p) + V.
	\end{equation*}
	In this equation, $A$ is the matrix \[2(I-pp\tp)\left[J_{d}(p)-2\mymatrix{P_{d-1}(p_1,p_2)\\ Q_{d-1}(p_1,p_2)}p\tp\right] -2p\tp\mymatrix{P_{d}(p_1,p_2)\\ Q_{d}(p_1,p_2)}I,\] 
	where $J_{d}$ denotes the Jacobian of the vector field $\left(\begin{array}{r}  P_{d}(p_1,p_2)\\ Q_{d}(p_1,p_2)\end{array}\right)$ and $V$ is a polynomial vector function in the variables $u$, $p$ and $u-p$ such that 
	\begin{equation*}
		V = \mathcal{O}\left(\mynorm{u-p}^2\right) \quad \mathrm{as } \;\; \mynorm{u-p} \ra 0.
	\end{equation*}
	$-2 p\tp\mymatrix{P_{d}(p_1,p_2)\\ Q_{d}(p_1,p_2)}= -2 \alpha$ is an eigenvalue of $A$ for the left eigenvector $p$. W.l.o.g.\ let $\alpha >0$. Then there is at least one trajectory $u^*(t)$ with $\lim\limits_{t\ra t_\mathrm{max}-}u^*(t) = p_\infty$ for a finite value $t_\mathrm{max} < \infty$. Thus, $u^*$ is a separatrix in the sense of Definition \ref{def:sep} that ends at the critical point $p$ at infinity.
\end{proof}

\begin{corollary}\label{cor:bd0Flow}
	Let $z'=f(z) = \sum_{k=0}^{d} \alpha_k z^k$ be a complex, one{\hyp}dimensional, polynomial flow with $\alpha_d = a_d \in \R, d\ge 1$. Then $p=\pm(1,0)$ is a fixed point at infinity and no (limit) cycle at infinity can exist. If $d$ is even, there are only finitely many fixed points at infinity.
\end{corollary}
\begin{proof}
	Theorem 2 in \cite{Roeder2003}, Lemma~\ref{lem:complexPoly} and Theorem~\ref{thm:ExistSep}.
\end{proof}

\begin{remark}
	It is necessary in Theorem~\ref{thm:ExistSep} to require that $d>1$. The example $f(z) = z, P(x,y) = x , Q(x,y) = y$ with solution trajectories $x(t) = x(0)\rme^t, y(t) = y(0)\rme^t$ shows that no separatrices are present. 
\end{remark}

\begin{example}[Global Phase Portrait of \texorpdfstring{$z^2+1$}{z*z+1}]
	We apply the theory of this section to investigate the phase space behaviour of the dynamical system 
	\begin{equation*}
	z' = f(z) = z^2+1; \quad z,f(z) \in \C
	\end{equation*}
	which can be regarded as two-dimensional, real ODE
	\begin{eqnarray*}
		x' =& x^2-y^2+1\\
		y' =& 2xy.
	\end{eqnarray*}

	Its solution is $z(t) = \tan(t+c),$ where the complex constant $c\in \C$ depends on the initial value. \Cref{fig:z2p1} shows the phase portrait of this ODE. 
	\begin{figure}[htp]
		\begin{center}
			\includegraphics[width=0.7\linewidth]{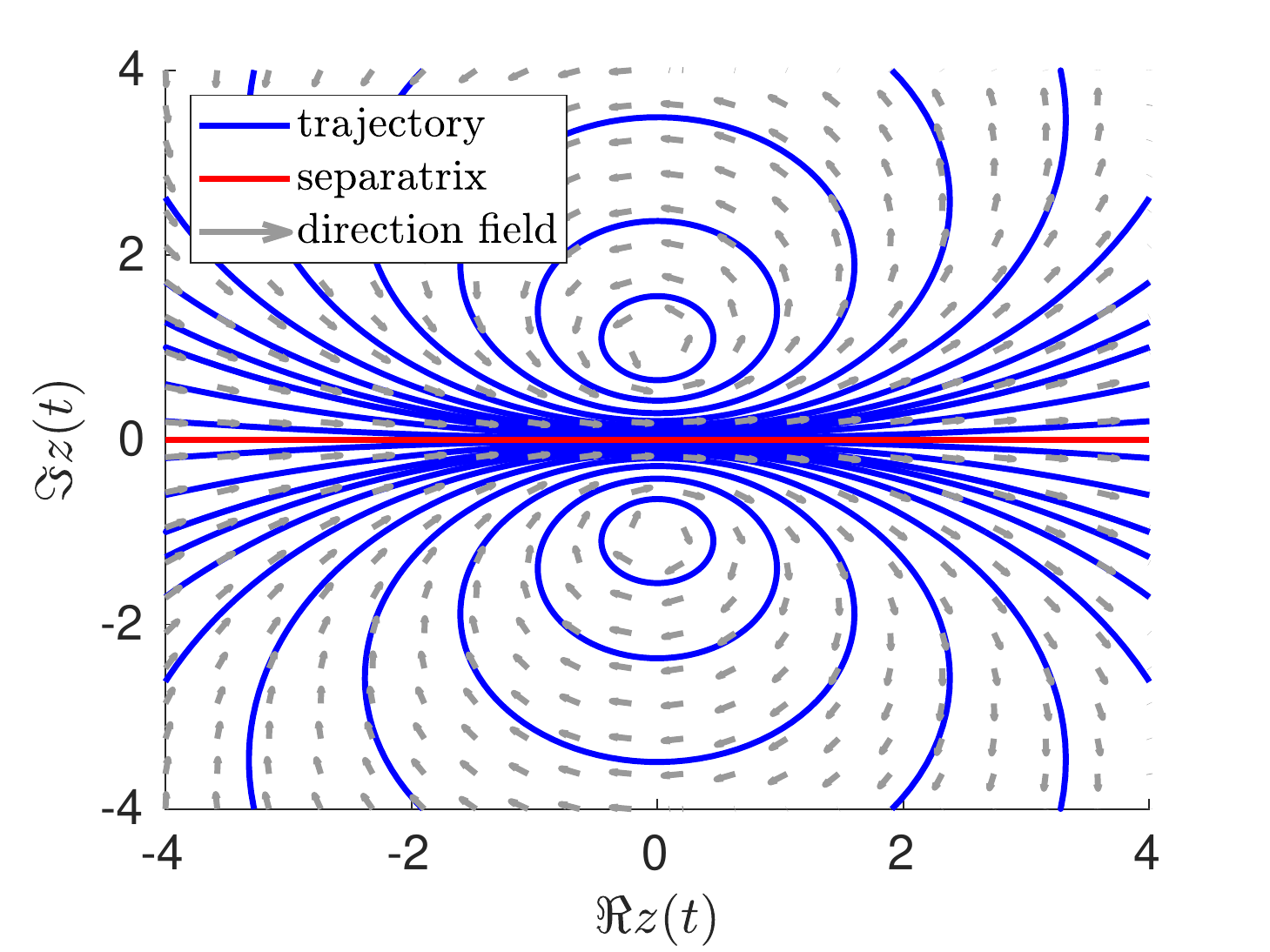}
			\caption{Holomorphic flow of $z'= z^2+1$. \enquote{Ordinary} trajectories in blue, separatrix in red, direction field (= normalized vector field) in grey.}\label{fig:z2p1}
		\end{center}
	\end{figure}

	The finite fixed points are given by $(0,1)$ and $(0,-1)$. The Jacobian
	\begin{equation*}
	\mymatrix[rr]{2x & -2y\\2y & 2x}
	\end{equation*}
	reveals that $(0,1)$ and $(0,-1)$ are centres of different signs. Some periodic orbits around these centres are plotted with orange colour in \cref{fig:z2p1PS}.
	
	Application  of Theorem~\ref{thm:ExistSep} guarantees the existence of a separatrix and according to Corollary~\ref{cor:bd0Flow} $p_{1,2}=\pm(1,0)$ are two (of finitely many) critical points at infinity. 
	Explicit calculation of the critical points at infinity is carried out by solving  
	\begin{equation*}
	0 = XQ_2(X,Y)-YP_2(X,Y) = Y(X^2+Y^2) = Y
	\end{equation*}
	and thus are given by $(X,Y,Z) = \pm (1,0,0)$. These two infinite critical points are depicted as green balls in \cref{fig:z2p1PS} and are the only two equilibria at infinity.
	
	\begin{figure}[ht]
		\centering
		\includegraphics[width=4cm]{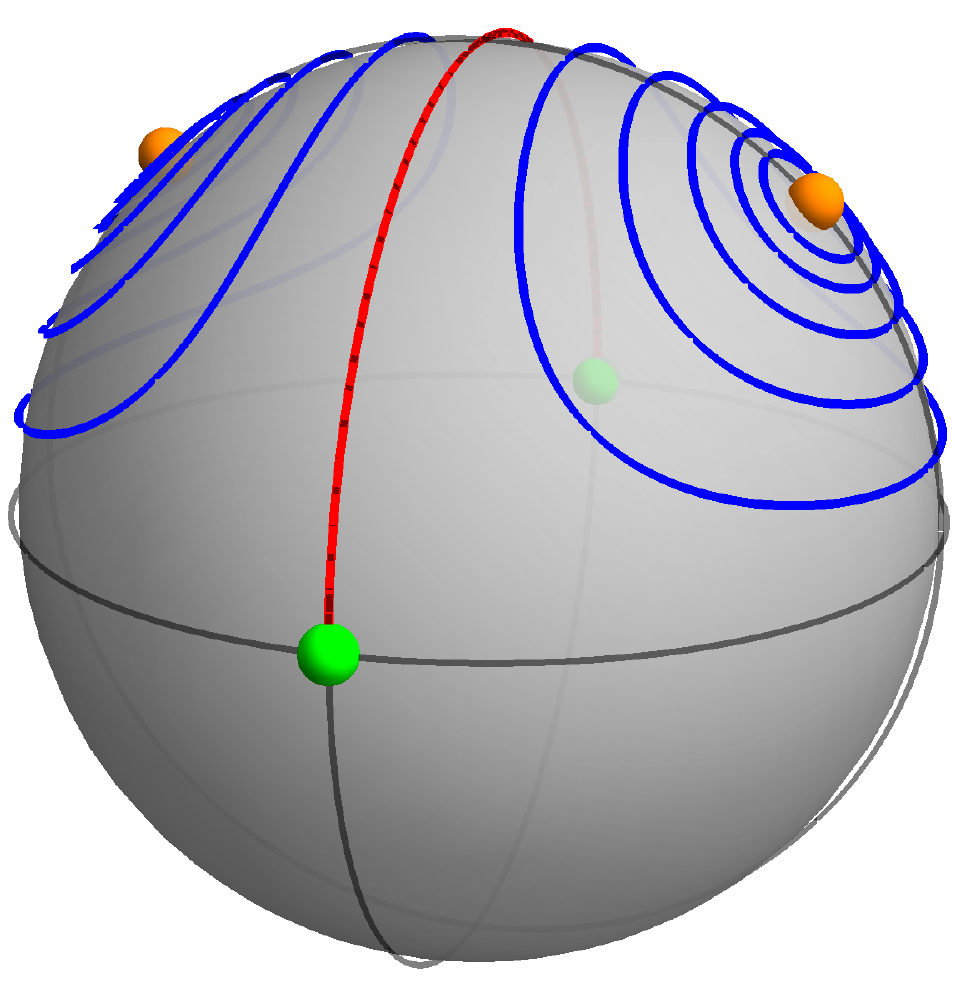}
		\caption{Global phase portrait of $z^2+1$ on the Poincar\'{e} Sphere.}\label{fig:z2p1PS}
	\end{figure}
	
	To determine the type of these equilibria at infinity, the transformed ODE
	\begin{equation*}
	\begin{array}{lll}
	\dot{\beta} &=\gamma^2 Q\left(\frac{1}{\gamma},\frac{\beta}{\gamma}\right)-\beta \gamma^2 P\left(\frac{1}{\gamma},\frac{\beta}{\gamma}\right) &= \beta + \beta^3-\beta \gamma^2  \\
	\dot{\gamma} &= -\gamma^{3} P\left(\frac{1}{\gamma},\frac{\beta}{\gamma}\right) &= -\gamma +\beta^2 \gamma - \gamma^3.
	\end{array}
	\end{equation*}
	is considered.
	The Jacobian evaluated at $(\beta,\gamma) = (0,0)$ is given by
	\begin{equation*}
	    \mymatrix[rr]{1 & 0\\0 & -1}.
	\end{equation*}
	Consequently, the critical point $(X,Y,Z) = (1,0,0)$ is saddle point. The equilibrium $(X,Y,Z) = (-1,0,0)$ is a saddle point with opposite direction of the flow because the degree $d=2$ of the regarded polynomials of the right{\hyp}hand side is an even number.
	For the sake of completeness, we would like to add that 
	\[\mymatrix{P_2(\pm1,0)\\ Q_2(\pm1,0)} = \mymatrix{1\\0} = \pm1 \cdot \mymatrix{\pm1\\0},\]
	i.e., $\alpha = \pm 1$. Thus, we have an incoming separatrix for $(1,0)$ and an outgoing separatrix for $(-1,0)$. 
	
	Next, an initial point for the numerical calculation of separatrices is given below. Obviously, an eigenvector for the stable eigenvalue $-1$ of  $(\beta,\gamma) = (0,0)$ is $e_2 = {(0,1)}\tp$. Perturbation of this fixed point into this direction leads to
	\[ p_\eps = {(0,0)}\tp + \eps e_2 = {(0,\eps)}\tp. \]
	Mapping this point back to the Poincar{\'e} sphere yields
	\[(X_\eps,Y_\eps,Z_\eps) = \left(\frac{1}{\sqrt{1+\eps^2}},0,\frac{\eps}{\sqrt{1+\eps^2}}\right).\] This points results by transformation back to the $(x,y)$ coordinates in the point
	$(x_\eps,y_\eps) = \left(\frac{1}{\eps},0\right)$. Since $y_\eps = 0$, the perturbed point is already on the separatrix and propagation in backward time directly calculates the separatrix numerically.
\end{example}

\begin{example}[Global Phase Portrait of $\cosh(z-1/2)$]
	The  theory of the transformation to the Poincar{\'e} sphere is not applicable for $\cosh$ since it is not a polynomial. However, the global phase portrait of  $z'=\cosh(z-1/2)$ can be described by using the phase portrait of $z' = z^2+1$ because the transformation $w(t)=\exp\big(z(2t)-1/2\big)$ turns every solution trajectory of $z'=\cosh(z-1/2)$ into a solution trajectory of $w'=w^2+1$.  Essentially, the inverse transformation is the logarithm. As a consequence, the global phase portrait of $z' = z^2+1$ is homeomorphic to the phase portrait of $z'=\cosh(z-1/2)$ in each strip $S_k := \left\{z \in \C : 2k\pi \le \Im(z) < 2(k+1)\pi \right\}$ for $k\in \Z$, where the logarithm is bijective.
\end{example}

\section{A Topological View on Separatrices of Periodic Orbits}\label{sec:topology}
In this section we extend the investigations of Broughan~\cite{Broughan2003, Broughan2003a, Broughan2005}. He analyses the neighbourhoods of a centre, elliptic sector at a zero, focus and nodes of holomorphic flows. He was able to show that the boundary components of these neighbourhoods are separatrices in the sense of Definition \ref{def:sep}.
In the following we restrict ourselves to the case of boundary components of two neighboured centres, i.e.\ where the neighbourhoods of two centres have non{\hyp}empty intersection (see \cref{fig:cosh}). We show that the index of trajectories around centres changes by traversing a common separatrix. 
But first, we define the term \emph{transit time}.

\begin{definition}[transit time (cf. Definition 3.2 in \cite{Broughan2003a})]
	Let $\gamma$ be an orbit of the holomorphic flow $z' = f(z)$. For $x,y \in \gamma$ the transit time $\tau(x,y)$ is defined by
	\begin{equation*}
		\tau(x,y) := \int_{x}^{y}\frac{1}{f(z)}\dif z,
	\end{equation*}	
	where the integration path is given by $\gamma$.
\end{definition}
Furthermore, the ball of radius $r>0$ around $z\in \C$ is denoted by $U_r(z)$.

The following Theorem~\ref{thm:IndexSep} characterises separatrices in holomorphic dynamical systems (not only polynomial systems), where two regions with periodic orbits inside have common boundary. The winding number of a periodic orbit around a centre is $\pm 1$ and has the same sign for all periodic orbits around the same centre. However, the periodic orbits on the other side of the separatrix (and thus around another centre) must have the opposite winding number.
\begin{theorem}\label{thm:IndexSep}
	Let $z_0,z_1$ be two (neighboured) centres of the flow \eqref{eq:ode} with a holomorphic function $f$. Let the regions 
	\begin{eqnarray*}
		P_k := \{z_k\} \cup &\left\{z^*\in \mathbb{C} \,: \mbox{ solution of (\ref{eq:ode}) with } z(0)=z^* \mbox{ is a} \right.\\ & \left.\mbox{periodic orbit around }z_k \,\right\}, \,\, k\in \left\{0,1\right\}
	\end{eqnarray*}
	have common boundary component $C \subset \partial P_0 \cap \partial P_1$ (i.e.\ $C$ is a separatrix, cf. Broughan~\cite{Broughan2003a}). Then it holds  $z^* \in C$
	iff for all $\varepsilon > 0$ small enough, there exists $\, z_k^* \in U_\varepsilon(z^*) \cap P_k,  \;\,  k\in \{0,1\}$ such that the orbits $\gamma_k$ of the flow (\ref{eq:ode}) with $z(0)=z_k^*$ satisfy $\mathrm{ind}_{\gamma_0}(z_0) \cdot \mathrm{ind}_{\gamma_1}(z_1) = -1.$
\end{theorem}

\begin{proof}
	If for all $\varepsilon > 0$ there are $z_k^*  \in U_\varepsilon(z^*) \cap P_k,  \;\,  k\in \{0,1\}$ then $z^* \in \partial P_0 \cap \partial P_1$ because $P_0$ and $P_1$ have empty intersection. According to Broughan~\cite{Broughan2003a} $z^*$ is thus part of a separatrix. 
	
	The interesting part of the proof is to show that the index of trajectories changes by traversing the separatrix. Let $z^*$ be part of a boundary component $C$ of $\partial P_0 \cap \partial P_1$. The separatrix $C$ is a trajectory. So for $\varepsilon > 0$ small enough there are two points $x,y \in U_{\varepsilon/2} (z^*) \cap C$ with $x\neq y$. Claim A in the proof of \cite[Theorem 4.1]{Broughan2003a} yields that 
		$\forall \delta > 0 \, \exists \, \varepsilon_\delta> 0$ such that all orbits $\alpha_k$  with 
	\begin{equation*}
		x_k' \in \alpha_k \cap U_{\varepsilon_\delta}(x) \cap P_k \neq \emptyset~\textrm{and}~y_k' \in \alpha_k \cap U_{\varepsilon_\delta}(y) \cap P_k \neq \emptyset 
	\end{equation*}
	 it holds
	 \begin{equation*}
		  	\left| \tau(x_k',y_k')- \tau(x,y)\right| < \delta \quad  \textrm{for } k \in \left\{0,1\right\}.
	 \end{equation*}
	 Thus, we have $\left| \tau(x_0',y_0')- \tau(x_1',y_1')\right| < 2\delta$. Since $\delta>0$ is arbitrary, $\tau(x_0',y_0')$ and $\tau(x_1',y_1')$ must have the same sign. If we take orbits $\alpha_k$ such that $x_k'$ resp. $y_k'$ have minimum distance to $x$ resp. $y$ then we can make use of the Jordan curve theorem in order to show that the interiors of $\alpha_k$ (which contain $z_k$) are to the left of $\alpha_0$ in direction of the flow and to the right of $\alpha_1$ or vice versa. In other words, 
	 \begin{equation*}
	 	\mathrm{ind}_{\alpha_0}(z_0) \cdot \mathrm{ind}_{\alpha_1}(z_1) = -1.
	 \end{equation*}
\end{proof}

\begin{example}
    We consider the function $f(z) = \cosh(z-1/2)$ again. \Cref{fig:cosh} includes the direction field of $f$. It is easy to see that the index of orbits in the strip $S_k := \left\{z  \in \C \, : \, \Im z \in \big(k\pi, (k+1)\pi\big) \right\}$ is ${(-1)}^{k}$. In accordance with Theorem~\ref{thm:IndexSep}, the index changes at each separatrix $\Im z = k\pi, k \in \Z$ for $k \in \Z$.
\end{example}

\section{Separatrices of 1-D holomorphic flows and Newton flows in complex time}\label{sec:complextime}
The works of Broughan~\cite{Broughan2003,Broughan2003a,Broughan2004,Broughan2005}, of Neuberger, Schleich et al.~\cite{Neuberger2014,Neuberger2015,Schleich2018}
study specific trajectories (separatrices) dividing the phase space of particular 1-D complex dynamical systems into distinct period orbit or fixed-point stability regions enclosing. Understanding the structure and asymptotics of the separatrices in the limit $\Re z \rightarrow \infty$ for the Newton flow of the Riemann $\xi$-function might yield insight into the location of the fixed points (corresponding to $\xi${\hyp}zeroes) via topology and geometry of the solution manifold, e.g.\ using arguments based on the constant phase of Newton flow lines as pointed out in \cite{Schleich2018}.
A \enquote{model flow} which has a phase portrait topologically similar to $\xi$ and is much more accessible is the complex $\cosh$. We propose on the basis of numerical results for the holomorphic $\cosh$ flow that the introduction of complex time may have the potential to establish a setting for more complete characterisation of separatrices including their analytical and topological properties. 

Recently we proposed complex time and analytic continuation of real{\hyp}analytic dynamical systems to study the properties of slow invariant manifolds (SIM), whose phase space similarities to separatrices have been pointed out above (in particular trajectory bundling and attraction/repulsion in positive, respectively negative time direction). It became obvious that spectral information can be gained by studying imaginary time trajectories of linear and nonlinear systems. Fourier transform even yields quantitative criteria characterizing SIM in several example systems (see Dietrich and Lebiedz~\cite{Dietrich2019}).

The solution trajectories in complex time are embedded Riemann surfaces (ho\-lo\-mor\-phic curves). For the complex $\cosh$, both topology and geometry of the Riemann surfaces as well as the fibre bundle of Riemann surfaces with base space $\C$ parametrising the possible initial values of (\ref{eq:ode}) seem to be closely related to properties of separatrices (cf.\ \cref{fig:coshNewtonFlow4D}). The Riemann surfaces (as graphs over compact subsets of complex time) for the complex time $\cosh$ flow with initial values that are not on a separatrix show some kind of \enquote{bifurcation} close to the fixed points (cosh zeroes). However, the Riemann surfaces with initial values on a separatrix do not show a bifurcation (see phase pace plots \cref{fig:coshNewtonFlow2D}). Non{\hyp}periodic trajectories with amplitude oscillations in imaginary time turn into periodic orbits at some bifurcation point. 


For the complex time ODE 
\[ \dd{t}z(t) = \dd{\tau_1+\rmi \tau_2}z(t) = f\big(z(t)\big), z_0 \in \C \]
with holomorphic function $f$, there is at least a local, unique solution $z(\cdot)$ representing a holomorphic curve \cite{Ilyashenko2007}. This solution satisfies the following partial differential equations:
\begin{equation}\label{eq:complextimeODE}
	\pp{\tau_1}z(t) = f\big(z(t)\big), \quad  \pp{\tau_2}z(t) = \rmi f\big(z(t)\big), z_0 \in \C
\end{equation}

\begin{figure}[ht]
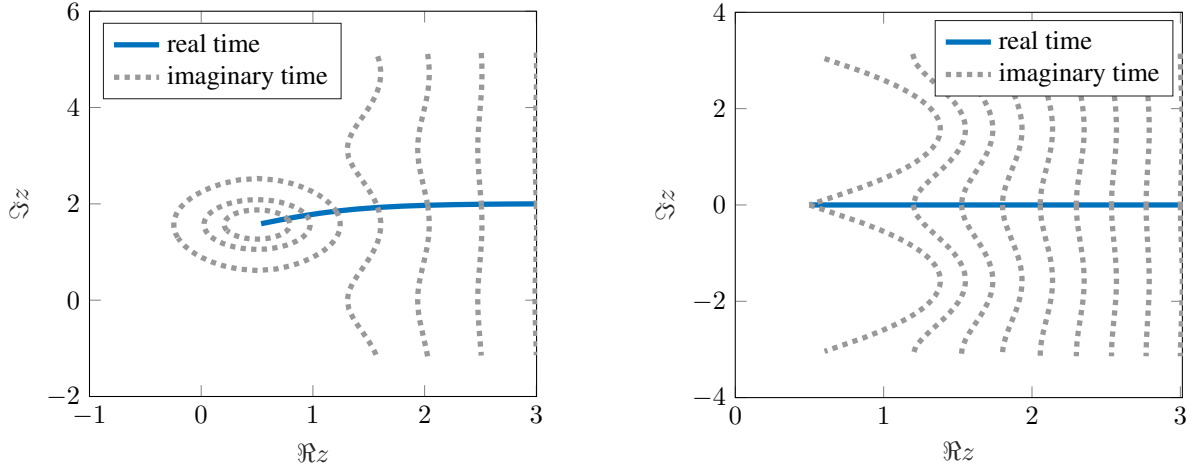

	\begin{center}
		\begin{minipage}{0.48\linewidth}
			\input{coshNFNoSep2D.tikz}
		\end{minipage}		
		\hfill
		\begin{minipage}{0.48\linewidth}
			\input{coshNFSep2D.tikz}
		\end{minipage}
		\caption{Flow of $\cosh$ in real time (blue trajectories) and imaginary time (grey trajectories). On the left figure, the initial value of the real trajectory is away from all separatrices while on the right figure, it is on a separatrix.}\label{fig:coshNewtonFlow2D}
	\end{center}
\end{figure}

\begin{figure}[ht]
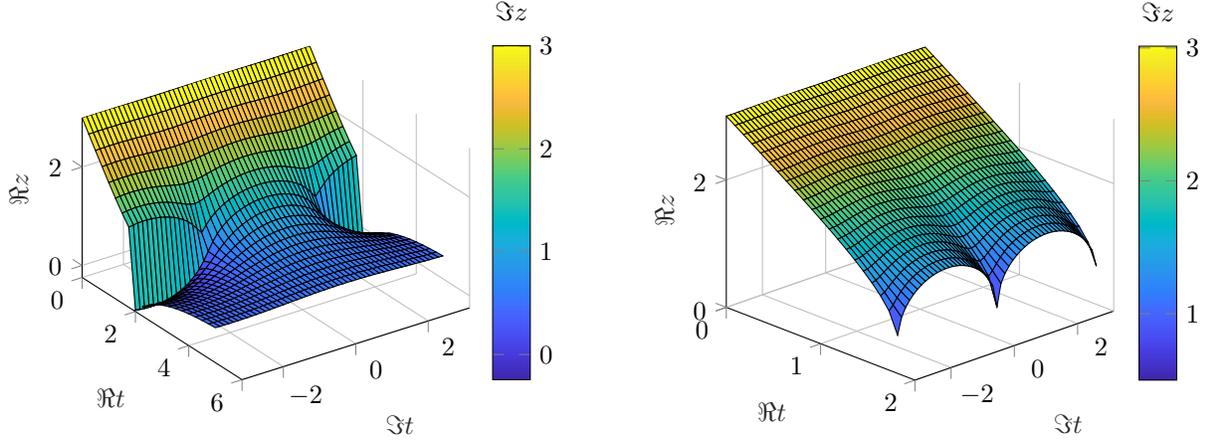

	\begin{center}
		\begin{minipage}{0.48\linewidth}
			\input{coshNFNoSep4D.tikz}
		\end{minipage}		
		\hfill
		\begin{minipage}{0.48\linewidth}
			\input{coshNFSep4D.tikz}
		\end{minipage}
		\caption{Graph of Riemann surface solutions $\big(t,z(t)\big) \in \C^2,$ where $z$ is solution of the differential equation $dz/dt = \cosh(z-1/2)$. In the left figure, the initial value $z(0)$ is away from all separatrices, while in the right figure $z(0)=3$ is on the separatrix.}\label{fig:coshNewtonFlow4D}
	\end{center}
\end{figure}
We think that this is an appropriate setting to study separatrices from a global point of view opening it for powerful tools from complex analysis, topology and algebra related to Riemann surfaces. 

Mappings of rectangles from the complex time plane along the complex{\hyp}time flow of the dynamical system to the solution manifold for a fixed initial value (see \cref{fig:coshRectangles}) allow to probe global topological properties of the Riemann surfaces which are by their complex time parameterization naturally coverings of compact subsets of $\mathbb{C}$ parameterised by the initial value $z_0$ of (\ref{eq:complextimeODE}).

\begin{figure}[htp]
	\begin{center}
		\input{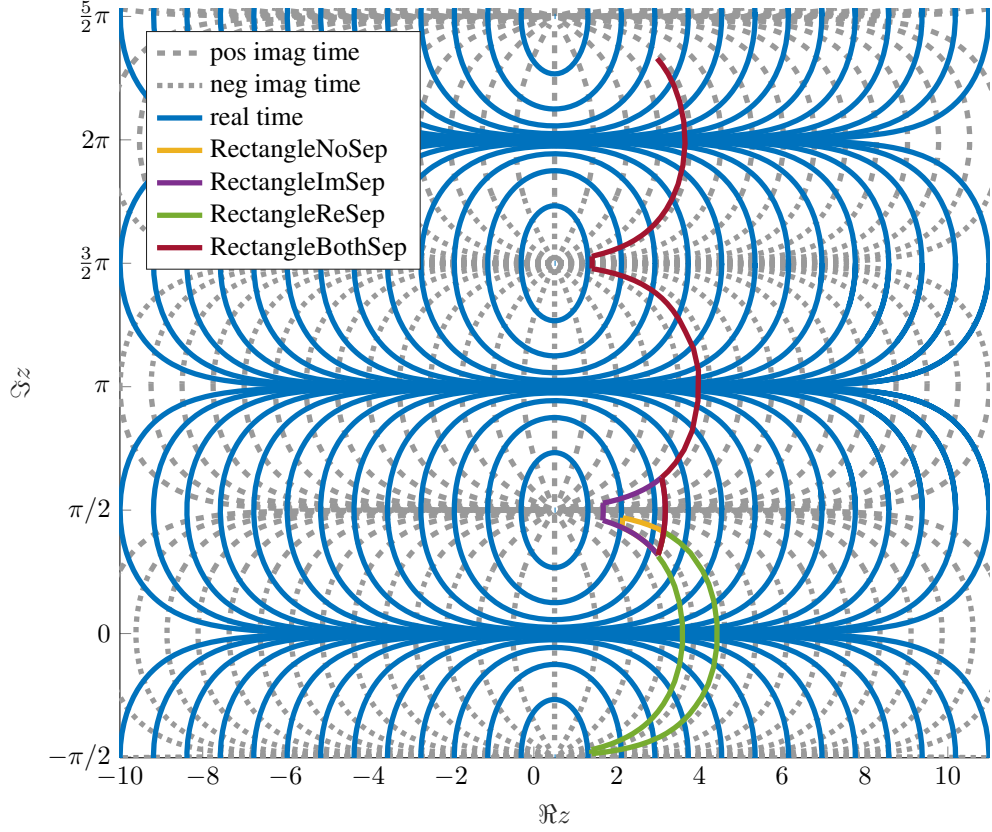}
		\caption{Plot of $\cosh$ trajectories in real and imaginary time. Four example rectangles from the complex plane are taken as \enquote{time{\hyp}input} for solution trajectories. Depending on whether they cross the real and the imaginary separatrix they yield closed curves on the solution manifold or not. NoSep: No separatrix is crossed, ImSep: Imaginary-Time Separatrix is crossed, ReSep: Real-Time Separatrix is crossed, BothSep: Both imaginary-time and real-time separatrices are crossed.}\label{fig:coshRectangles}
	\end{center}
\end{figure}

\subsection{Newton Flows of \texorpdfstring{$\cosh$}{cosh} and \texorpdfstring{$\xi$}{xi}}
Newton flow lines $z(t)$ (solutions of (\ref{eq:newtonflow})) share the characteristic property of constant phase of the complex number $f(z(t))$ along solutions $z(t)$ and can be interpreted as gradient flow lines with respect to a Sobolev gradient defined by a Riemannian metric induced by the vector field \cite{Neuberger1999}.
Lines of constant modulus $|f(z(t)|$ are orthogonal to Newton flow lines. In a complex time view (see (\ref{eq:complextimeODE})), the latter correspond to imaginary time trajectories. Newton flows can be desingularised removing the singularities at zeroes of $f'$ by rescaling time which results in the same flow lines with different time parameterisation \cite{Jongen1988,Benzinger1991}.
The singularities of the Newton flow turn into hyperbolic fixed points (saddles). This brings them close to earlier considerations of this article studying saddles at infinity (see \cref{sec:infinity}). 

In particular, the inverted Newton flow of the Riemann $\zeta${\hyp}function occurs in the Riemann{\hyp}von{\hyp}Mangoldt explicit formula with $\zeta${\hyp}zeroes $\rho_n$:

\begin{equation*}
    \psi_0(x) = \frac{1}{2\pi \rmi}\int\limits_{\sigma-\rmi \infty}^{\sigma +\rmi \infty}\left(-\frac{\zeta'(z)}{\zeta(z)}\right) \frac{x^z}{z} \dif z = x -\sum\limits_n \frac{x^{\rho_n}}{\rho_n} - \log(2\pi) -\frac{1}{2} \log(1-x^{-2}),
\end{equation*}
where
\begin{equation*}
    \psi(x) = \sum\limits_{p^k \le x} \log p \quad \mathrm{and} \quad \psi_0(x) = \frac{1}{2} \lim\limits_{h\rightarrow 0}\left(\psi(x+h)+\psi(x-h)\right)
\end{equation*}
with the second Chebyshev function $\psi$.

This was the original motivation for the authors to study a dynamical system based on $\zeta$ Newton flow and to introduce complex time. The Riemann{\hyp}von{\hyp}Mangoldt explicit formula contains the integral over the inverted $\zeta$ Newton flow in imaginary direction and relates this integral to a sum over $\zeta${\hyp}zeroes. Thus the phase portrait of the $\zeta$ Newton flow, respectively \emph{symmetric} $\xi$ flow, should contain some information on the $\xi$-zeroes. 

We focus on the Newton flow of the symmetric (cf.\ \emph{functional equation}), entire $\xi${\hyp}function whose phase portrait has characteristic similarities to the complex cosh Newton flow.

\begin{equation*}
    \xi(z) := \frac{1}{2}z(z-1)\Gamma\left(\frac{z}{2}\right)\pi^{-z/2} \zeta(z).
\end{equation*}

Using the logarithmic derivative of the formula (see \cite[p.~47]{Edwards2001}) $\xi(z)=\xi(0) \prod\limits_{n}(1-\frac{z}{\rho_n})$ and separation of variables for the Newton Flow \eqref{eq:newtonflow} of $\xi(z)$ 
we get 
    \begin{align*}
	& & -\frac{\xi'(z)}{\xi(z)} \dif z&=1 \cdot \dif t, & \left| \quad \int_{} \right.\\
	&\Leftrightarrow  & -\int\limits_{z_0}^{z(T)}\frac{\xi'(z)}{\xi(z)} \dif z &= \int\limits_{0}^{T} 1 \dif t, & z(0)=z_0\\
	&\Leftrightarrow & -\int\limits_{z_0}^{z(T)} \sum\limits_{n} \frac{1}{z-\rho_n}\dif z &= T &\\
	&\Leftrightarrow & -\sum\limits_{n}\int\limits_{z_0}^{z(T)} \frac{1}{z-\rho_n}\dif z &= T &\\
	&\Leftrightarrow & \sum\limits_{n}\log(z_0-\rho_n) - \sum\limits_{n} \log(z(T)-\rho_n) &\equiv T & (\mathrm{mod}\; 2\pi \rmi)\\
	&\Leftrightarrow & \sum\limits_{n} \log \frac{z(T)-\rho_n}{z_0-\rho_n} &\equiv -T & (\mathrm{mod}\; 2\pi \rmi) \\
	&\Leftrightarrow & \exp\left(\sum\limits_{n} \log \frac{z(T)-\rho_n}{z_0-\rho_n}\right) & \equiv \rme^{-T} & (\mathrm{mod}\; 2\pi \rmi) \\
	&\Leftrightarrow & \prod\limits_{n} \frac{z(T)-\rho_n}{z_0-\rho_n} &\equiv \rme^{-T} & (\mathrm{mod}\; 2\pi \rmi)
    \end{align*}
When the integration path encloses zeroes of $\xi$, the residue theorem has to be applied which brings addition of $k2\pi \rmi$. 
Finite approximation of the infinite product of $\xi${\hyp}zeroes allows the definition of a polynomial
\begin{equation*}
    P_m(z;T,z_0):= \prod\limits_{n=1}^{m} \frac{z-\rho_n}{z_0-\rho_n} - \rme^{-T}.
\end{equation*}
An alternative derivation of $P_m(z;T,z_0)=$ is to combine the product formula
\begin{equation*}
    \xi(z) = \xi(z_0)\prod\limits_{n} \frac{z-\rho_n}{z_0-\rho_n}
\end{equation*}
with a line of constant phase as solution trajectory of the Newton flow of $\xi$
\begin{equation*}
    \xi\big(z(T)\big) = \xi(z_0)\rme^{-T}.
\end{equation*}
This also yields 
\begin{equation*}
    \rme^{-T} = \prod\limits_n \frac{z(T)-\rho_n}{z_0-\rho_n},
\end{equation*}
with approximating polynomials $P_m\big(z;T,z_0\big)$.

Contributions to the Newton flow of those terms corresponding to $\xi$ zeros \enquote{far away} from the initial value should be small for $|T|$ not too large.
The complex value of the solution trajectory at the specific time $T \in \C$ starting at $z(0)=z_0$ can then by approximated by the polynomial equation
\begin{equation}\label{eq:polynomapproxXi}
    P_m(z;T;z_0) = 0,
\end{equation}
which defines a smooth complex algebraic curve.
The approximated Newton flow vector field is in this case a rational function and can be transformed to a polynomial by rescaling time as in \cite{Benzinger1991}. 
Thus, the approach to address fixed points at infinity and potentially associated separatrices presented in this article (\cref{sec:infinity}) might be applicable to study the asymptotics of separatrices in polynomial approximations of the Newton-flow of $\xi$. Figures~\ref{fig:xiApprox4} and \ref{fig:xiApprox40} illustrate that with increasing number of complex conjugated $\xi$-zeroes the polynomial flow solution manifolds comes closer to the $\xi$-Newton flow phase portrait in \cref{fig:xiNewtonFlow}. An approximation theoretical analysis is subject to future research. 
Figures~\ref{fig:xiApprox4} and \ref{fig:xiApprox40} were created with MATLAB. The first 100 $\xi${\hyp}zeroes of Odlyzko \cite{Odlyzko} were used for numerical calculations. Eq.\ \eqref{eq:polynomapproxXi} is solved for discrete time points $T$ for each initial value $z_0$ on a discretised grid of the compact time set $[-7,8] \times [-1,30]\rmi$. \cref{fig:xiNewtonFlow} shows the phase portrait of the Newton flow of $\xi$. It was calculated using Mathematica.

\begin{figure}
    \begin{minipage}{0.48\linewidth}
        \centering
        \input{xiApproxDeg4.tikz}
        \caption{$\xi$ Newton flow ap\-prox\-i\-ma\-tion using \eqref{eq:polynomapproxXi}, $m=4$ zeroes of $\xi$ (first two zeroes and its complex conjugates).}\label{fig:xiApprox4}
    \end{minipage}
    \hfill
    \begin{minipage}{0.48\linewidth}
	\centering
	\input{xiApproxDeg40.tikz}
	\caption{$\xi$ Newton flow ap\-prox\-i\-ma\-tion using \eqref{eq:polynomapproxXi}, $m=40$ zeroes of $\xi$ (first 20 zeroes and its complex conjugates).}\label{fig:xiApprox40}
    \end{minipage}
\end{figure}   
\begin{figure}
    \centering
    \includegraphics[width=0.4\linewidth]{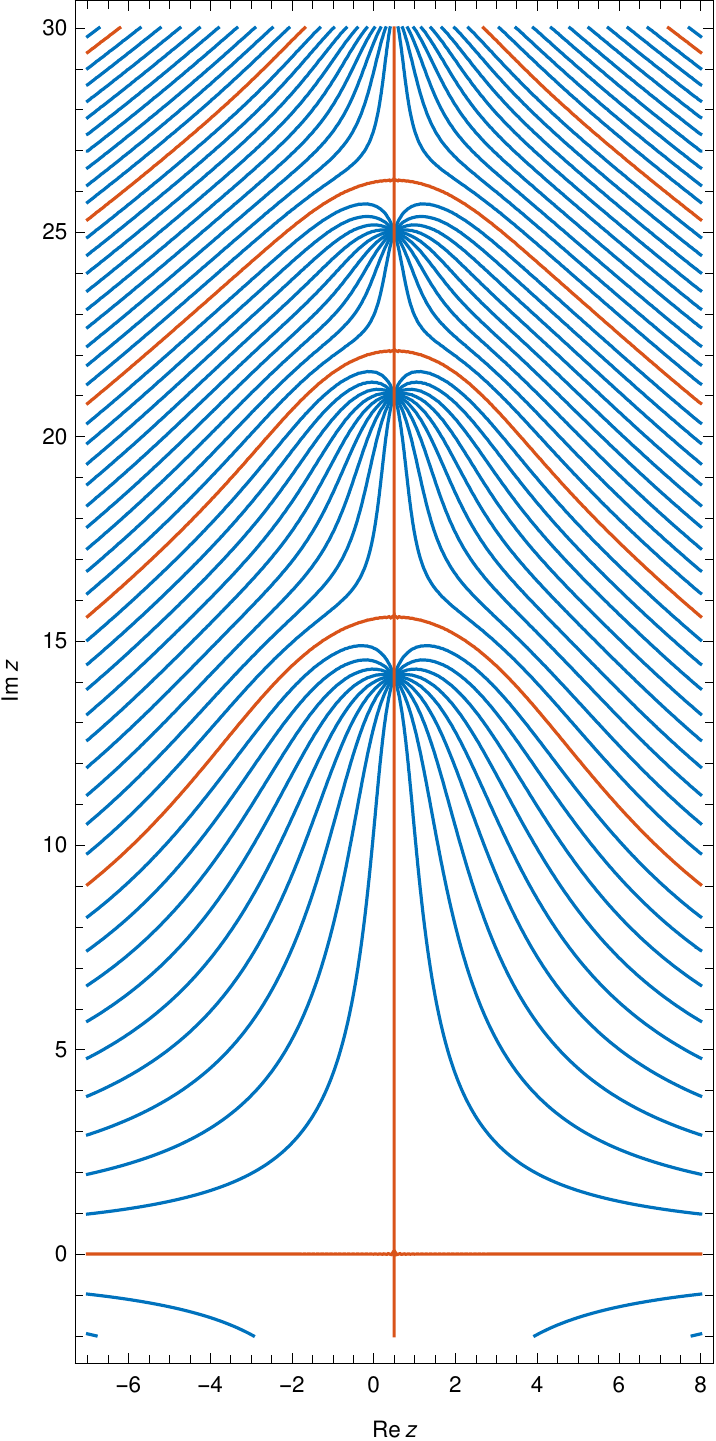}
    \caption{Newton flow of $\xi$. Trajectories are plotted by calculating lines of constant phase. Horizontal red lines are separatrices.}\label{fig:xiNewtonFlow}
\end{figure}   

If $z(\cdot)$ is a periodic orbit of the Newton flow of $\xi$ in complex time with \enquote{complex period} $T \in \C$ a necessary condition for $z(T)=z_0 \in \C$ is
\begin{equation*}
    \rme^{-T} = \prod\limits_{n} \frac{z(T)-\rho_n}{z_0-\rho_n} = 1.
\end{equation*}
Therefore, $T$ must be a multiple of $2\pi \rmi$, i.e.\ $T = 2\pi k \rmi, k \in \Z$. 

We think that such polynomial approximations of the $\xi$ Newton flow based on an explicit formula for the $\xi${\hyp}zeroes might allow to build a bridge into the world of algebraic geometry and corresponding approaches to the Riemann $\xi${\hyp}function and maybe even allow a view on the Riemann hypothesis. There might also be a relation to the Hilbert{\hyp}Polya conjecture associating the $\xi${\hyp}zeroes with eigenvalues of some (self{\hyp}adjoint) operator. Note that the Riemann dynamics, a hypothetic classical chaotic dynamical systems which should lead to the Hilbert{\hyp}Polya operator via quantization, is supposed to have complex periodic orbits with periods $T_k=k\pi \rmi, k \in \N$ \cite{Sierra11,Berry99} which is not the case for the imaginary time trajectories of the $\xi$ Newton flow. Due to successive Riemann surface bifurcations of the Newton flow at zeroes of $\xi'$, the dynamical systems in complex time should be chaotic in some sense with a specific bifurcation{\hyp}related route to chaos.
We speculate that an interesting operator might be constructed via the Newton flow map of $\xi$ which seems to be related to a Hamiltonian system (the modulus $|\xi(z)|$ is a conserved quantity along imaginary{\hyp}time trajectories of the Newton flow and the phase of the complex number $\xi(z)$ along real{\hyp}time trajectories!). In the physical literature a specific formal relation between statistical mechanics and quantum mechanics is referred to as Wick rotation \cite{Wick54}, sometimes called a mysterious connection: when replacing time in the evolution operator $\rme^{-\frac{\rmi}{\hbar}Ht}$ of some Hamiltonian $H$ (from time{\hyp}dependent Schr\"odinger equation) by negative imaginary time $t \rightarrow -\rmi t$ and then $\frac{t}{\hbar}$ by $\frac{1}{T}$ one ends up with the time{\hyp}independent state density operator $\rme^{-\frac{H}{T}}$ from equilibrium statistical mechanics. We propose that this view might be of particular relevance in our context. Looking at \cref{fig:xiApproxImag6} and \cref{fig:xiApproxImag20} it becomes obvious how cyclic imaginary time trajectories are transported in real time direction along the Newton flow. When crossing in this time evolution a zero of the derivative $P_m(z;T,z_0)$ on a separatrix, the Riemann solution surface in complex time bifurcates and in imaginary time direction two cyclic branches occur, one enclosing a single zero and the other enclosing all remaining zeroes. This is reminiscent of a phase transition. 

In addition to the obvious algebraic issue to study the interesting distribution of zeroes of the polynomials $P_m(z;T,z_0)$ (red dots in \cref{fig:xiApproxImag6}, \cref{fig:xiApproxImag20}) as a function of the polynomial degree and their behaviour at Riemann surface bifurcation points, a sketch of more general potential relations to algebraic geometry might look as follows, which is, however, subject to current work: Compact Riemann surfaces $\mathcal{R}$ are biholomorphically equivalent to smooth complex projective algebraic curves. The latter are also defined as curves over some number field under the conditions of Belyi's theorem \cite{Belyi79, Belyi02}, i.e.\ a Riemann surface is defined as an algebraic curve over a finite field if and only if there is a non{\hyp}constant meremorphic function on $\mathcal{R}$ which branches in at most three points, i.e.\ $\mathcal{R}$ is a curve defined over the algebraic numbers $\bar{\mathbb{Q}}$  if and only if a morphism $\mathcal{R} \rightarrow \mathbb{P}^1_{\C}$ with at most three critical values exists. This establishes a link between topological coverings and field extensions. With a view on the bifurcations/ramifications of Riemann surfaces pointed out above in the context of complex{\hyp}time Newton flows (see \cref{fig:coshNewtonFlow2D} and \cref{fig:coshNewtonFlow4D}), it might turn out to be particularly fruitful to look at these bifurcations in the context of Belyi's theorem. 

Another interesting direction could be to establish a link between the dynamical system that we suggested, i.e.\ the $\xi$ Newton flow in complex time and its polynomial approximations based on the $\xi${\hyp}zeroes to the work of Deninger \cite{Deninger1998} investigating analogies between number theory and dynamical systems.

\begin{figure}
    \begin{minipage}{0.47\linewidth}
        \centering
	\includegraphics[width=0.9\linewidth]{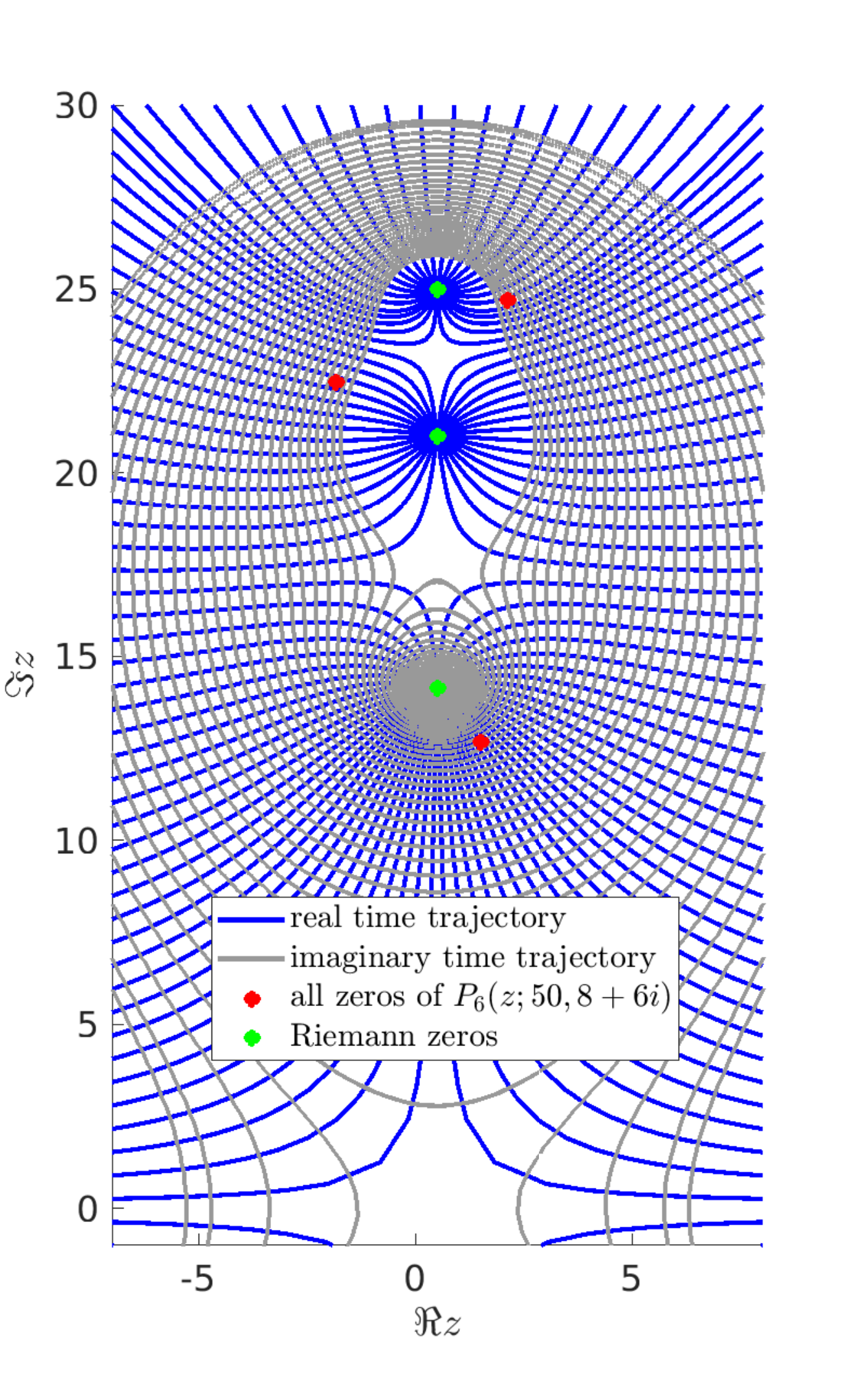}
	\caption{Real and imaginary time $\xi$ Newton flow solution ap\-prox\-i\-mated by using \eqref{eq:polynomapproxXi}, $m=6$ zeroes of $\xi$ (first three zeroes and their complex conjugates).}\label{fig:xiApproxImag6}
    \end{minipage}
    \hfill
    \begin{minipage}{0.47\linewidth}
	\centering
	\includegraphics[width=0.9\linewidth]{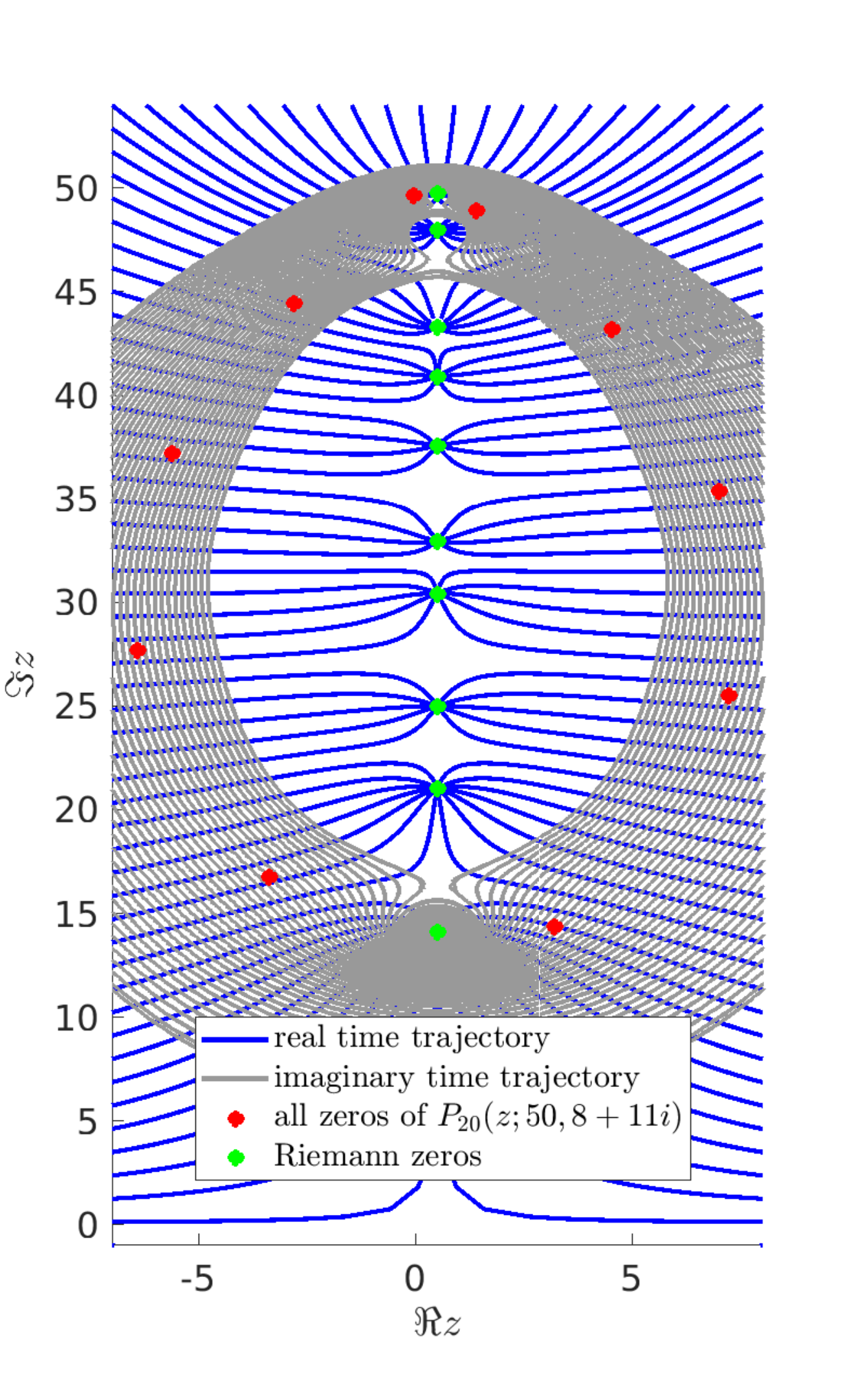}
	\caption{Real and imaginary time $\xi$ Newton flow solution ap\-prox\-i\-mated by using \eqref{eq:polynomapproxXi}, $m=20$ zeroes of $\xi$ (first 10 zeroes and their complex conjugates).}\label{fig:xiApproxImag20}
    \end{minipage}
\end{figure}   

\section{Summary and Outlook}
Central objects of this article are separatrices in holomorphic dynamical systems. Analytical results for polynomial systems are that fixed points at infinity are separatrix generators in that sense that for each critical point at infinity there is at least one separatrix. Moreover, we could reveal some conditions under which a 1{\hyp}D holomorphic complex dynamical system (that cannot have finite limit cycles according to Broughan; cf.\ \cite{Broughan2003}) also cannot have infinite ones.  

A quantifiable topological property of a separatrix, which is the common boundary of periodic orbits around two centres, is that the index of the orbits changes its sign when crossing a separatrix. This is a necessary and sufficient condition for separatrices in these systems. 

The following list contains ideas and remarks for future research.
\be[(a)] 
	\item 
		In this work only one{\hyp}dimensional complex dynamical systems were analysed. R{\"o}hrl and Walcher~\cite{Roehrl1997} may yield deeper insights for higher dimensional systems. The equivalent definition of Proposition~\ref{prop:InfPointCrit} \ref{prop:LinDep} for critical points at infinity is used in Proposition (1.7) of \cite{Roehrl1997} as sufficient conditions for equilibria of the dynamical system. 
	\item 
		Statements for holomorphic right{\hyp}hand sides $f(z)$ that are not polynomials or directly transformed into a polynomial like $f(z) = \cosh(z)$ are subject to current work. 
	\item 
		The reader may suspect that similar results than in this article could be obtained for meromorphic functions. Especially, the Newton flow of polynomial systems has high potential to be well understood and it may yield a possibility to understand separatrices and the Newton flow of $\xi$ better.
	\item We are convinced that introducing complex time $t = \tau_1 + \rmi \tau_2$ and analysing Riemann surface solutions may generate a setting to address general properties of separatrices, in particular for the $\xi$ Newton flow. We presented central observations and pointed our potential routes to follow in \cref{sec:complextime}.
\ee

\section*{Acknowledgments}
We gratefully acknowledge the Klaus\hyp Tschira\hyp Stiftung for financial support. Furthermore, the authors thank Jörn Dietrich, Johannes Poppe and Wolfgang Schleich for discussions on this topic.

\bibliography{literature.bib}

\end{document}